\documentclass[11pt,final]{article}%
\usepackage{ifdraft}
\usepackage{amsfonts}
\usepackage{amsmath}
\usepackage{amssymb}
\usepackage{amsxtra}
\usepackage[amsmath,hyperref,thmmarks]{ntheorem}
\usepackage{stix2}
\usepackage{makeidx}
\usepackage{graphicx}
\usepackage{cprotect}
\usepackage{natbib}
\usepackage[colorlinks=true,bookmarksnumbered=true,pdfpagemode=None,final]%
{hyperref}%
\providecommand{\U}[1]{\protect\rule{.1in}{.1in}}
\ifdraft{}{\usepackage{version}}
\ifdraft{}{\excludeversion{nt}}
\newcommand{\df}{\smash{\lower.12em\hbox{\textup{\tiny def}}}}

\def\1{{1\mkern-7mu1}}

\usepackage[overload]{textcase}

\usepackage{xcolor}
\definecolor{bblue}{rgb}{0.0, 0.0, 0.6}

\usepackage{array}

\usepackage{microtype}

\usepackage{tikz}
\usepackage{tikz-cd}
\usetikzlibrary{matrix,arrows,positioning,decorations.pathmorphing}
\usetikzlibrary{decorations.markings,shapes.geometric}
\tikzset{commutative diagrams/column sep/Huge/.initial=24ex}
\tikzset{commutative diagrams/column sep/tiny+/.initial=1.7ex}
\tikzcdset{arrow style=math font}

\usepackage{enumitem}
\setitemize[1]{label=$\diamond$\hspace{0.07in}}
\setlist{nolistsep}
\usepackage{url}
\usepackage{verbatim}

\usepackage[notcite,notref]{showkeys}

\usepackage{mathtools}

\usepackage{titletoc,titlesec}
\contentsmargin{2.0em}
\dottedcontents{section}[2.3em]{}{1.8em}{1pc}
\usepackage[nottoc]{tocbibind}

\titleformat*{\section}{\Large\itshape}
\titleformat*{\subsubsection}{\scshape}%\large
\titleformat*{\paragraph}{\itshape}

\usepackage{natbib}
\let\cite\citealt
\hypersetup{linkcolor=bblue,anchorcolor=bblue,citecolor=bblue,filecolor=bblue,urlcolor=bblue}
\hypersetup{pdftitle={Hodge classes on abelian varieties},pdfauthor={J.S. Milne}}
\hypersetup{pdfsubject={Hodge classes},pdfkeywords={abelian varieties}}

\usepackage[textwidth=5.5in,textheight=9.1in,centering,a4paper]{geometry}

\newcommand{\bcomment}{\begin{comment}}\newcommand\ecomment{\end{comment}}
\newcommand{\bfootnotesize}{\begin{footnotesize}}\newcommand\efootnotesize{\end{footnotesize}}
\newcommand{\bquote}{\begin{quote}}\newcommand\equote{\end{quote}}
\newcommand{\bsmall}{\begin{small}}\newcommand\esmall{\end{small}}
\newcommand{\btable}{\begin{table}}\newcommand{\etable}{\end{table}}
\newcommand{\dstyle}{\displaystyle}
\newcommand{\edocument}{\end{document}}

\newcommand{\textnf}{\textnormal}

\renewcommand{\Gamma}{\varGamma}
\renewcommand{\Pi}{\varPi}
\renewcommand{\Sigma}{\varSigma}

\DeclareMathOperator{\End}{End}

\DeclareMathOperator{\Gal}{Gal}
\DeclareMathOperator{\GL}{GL}
\DeclareMathOperator{\Hom}{Hom}

\DeclareMathOperator{\Tr}{Tr}
\newcommand{\Mot}{\mathsf{Mot}}%motives

\theoremnumbering{arabic}
\theoremheaderfont{\scshape}
\RequirePackage{latexsym}
\theorembodyfont{\slshape}
\theoremseparator{.}
\newtheorem{proposition}{Proposition}
\newtheorem{theorem}{Theorem}

\theorembodyfont{\upshape}
\newtheorem{plain}{}[section]
\newtheorem{remark}{Remark}
\theorembodyfont{\normalsize}
\theoremstyle{nonumberplain}
\newtheorem{proof}{Proof}
\theoremsymbol{\ensuremath{_\Box}}
\qedsymbol{\ensuremath{_\Box}}
\makeindex
\begin{document}

\title{Hodge classes on abelian varieties}
\author{J.S. Milne}
\maketitle

\begin{abstract}
We prove, following Deligne and Andr\'{e}, that the Hodge classes on abelian
varieties of CM-type can be expressed in terms of divisor classes and split
Weil classes, and we describe some consequences. In particular, we show that
 Grothendieck's standard conjecture of Lefschetz type implies the Hodge
conjecture for abelian varieties (Abdulali, Andr\'e, \ldots). No new results,
but the proofs are shorter.

\end{abstract}
\tableofcontents

\section{Review of abelian varieties of CM-type}

\begin{plain}
\label{r1}A complex abelian variety is said to be of CM-type if $\End^{0}(A)$
contains a CM-algebra\footnote{That is, a product of CM-fields.} $E$ such that
$H^{1}(A,\mathbb{Q}{})$ is free of rank $1$ as an $E$-module. Let
$S=\Hom(E,\mathbb{C}{})$, and let $H^{1}(A)=H^{1}(A,\mathbb{C}{})$. Then%
\[
H^{1}(A)\simeq H^{1}(A,\mathbb{Q}{})\otimes\mathbb{C}{}=\bigoplus_{s\in
S}H^{1}(A)_{s},\quad H^{1}(A)_{s}\overset{\df}{=}H^{1}(A)\otimes
_{E,s}\mathbb{C}{}\text{.}%
\]
Here $H^{1}(A)_{s}$ is the (one-dimensional) subspace of $H^{1}(A)$ on which
$E$ acts through $s$. We have%
\[
H^{1,0}(A)=\bigoplus_{s\in\Phi}H^{1}(A)_{s},\quad H^{0,1}(A)=\bigoplus
_{s\in\bar{\Phi}}H^{1}(A)_{s}\text{,}%
\]
where $\Phi$ is a CM-type on $E$, i.e., a subset of $S$ such that
$S=\Phi\sqcup\bar{\Phi}$. Every pair $(E,\Phi)$ consisting of a CM-algebra $E$
and a CM-type $\Phi$ on $E$ arises in this way from an abelian variety.
Sometimes we identify a CM-type with its characteristic function $\phi\colon
S\rightarrow\{0,1\}$.
\end{plain}

\begin{plain}
\label{r4}Let $A$ be a complex abelian variety of CM-type, and let $E$ be a
CM-subalgebra of $\End^{0}(A)$ such that $H^{1}(A,\mathbb{Q}{})$ is a free
$E$-module of rank $1$. Let $F$ be a Galois extension of $\mathbb{Q}$ in
$\mathbb{C}{}$ splitting\footnote{That is, such that $E\otimes_{\mathbb{Q}}F$
is isomorphic to a product of copies of $F$.} $E$, and let $S=\Hom(E,F)$. We
regard the CM-type $\Phi$ of $A$ as a subset of $S$. Let $H^{r}(A)=H^{r}%
(A,F)$. Then%
\[
H^{1}(A)\simeq H^{1}(A,\mathbb{Q}{})\otimes_{\mathbb{Q}{}}F=\bigoplus
\nolimits_{s\in S}H^{1}(A)_{s},\quad
\]
where $H^{1}(A)_{s}\overset{\df}{=}H^{1}(A,\mathbb{Q}{})\otimes_{E,s}F$ is the
(one-dimensional) $F$-subspace of $H^{1}(A)$ on which $E$ acts through $s$.

\begin{enumerate}
\item We have
\[
H^{r}(A)\simeq\bigwedge\nolimits_{F}^{r}H^{1}(A)=\bigoplus\nolimits_{\Delta
}H^{r}(A)_{\Delta}\quad\quad\text{(}F\text{-vector spaces),}%
\]
where $\Delta$ runs over the subsets of $S$ of size $|\Delta|=r$ and
$H^{r}(A)_{\Delta}\overset{\df}{=}\bigotimes_{s\in\Delta}H^{1}(A)_{s}$ is the
(one-dimensional) subspace on which $a\in E$ acts as $\prod\nolimits_{s\in
\Delta}s(a)$.

\item Let $H^{1,0}=\bigoplus\nolimits_{s\in\Phi}H^{1}(A)_{s}$ and
$H^{0,1}=\bigoplus\nolimits_{s\in\bar{\Phi}}H^{1}(A)_{s}$. Then
\[
H^{p,q}\overset{\df}{=}\bigwedge\nolimits^{p}H^{1,0}\otimes\bigwedge
\nolimits^{q}H^{0,1}=\bigoplus\nolimits_{\Delta}H^{p+q}(A)_{\Delta}\quad
\quad\text{(}F\text{-vector spaces),}%
\]
where $\Delta$ runs over the subsets of $S$ with $|\Delta\cap\Phi|=p$ and
$|\Delta\cap\bar{\Phi}|=q$.

\item (\cite{pohlmann1968}, Theorem 1.) Let $B^{p}=H^{2p}(A,\mathbb{Q}{})\cap
H^{p,p}$ ($\mathbb{Q}{}$-vector space of Hodge classes of degree $p$ on $A$).
Then
\[
B^{p}\otimes F=\bigoplus\nolimits_{\Delta}H^{2p}(A)_{\Delta}\text{,}%
\]
where $\Delta$ runs over the subsets of $S$ with%
\begin{equation}
|(t\circ\Delta)\cap\Phi|=p=|(t\circ\Delta)\cap\bar{\Phi}|\text{ for all }%
t\in\Gal(F/\mathbb{Q}{}). \label{eq2}%
\end{equation}

\end{enumerate}
\end{plain}

\section{Review of Weil classes}

\begin{plain}
\label{r2}Let $A$ be a complex abelian variety and $\nu$ a homomorphism from a
CM-field $E$ into $\End^{0}(A)$. The pair $(A,\nu)$ is said to be of Weil type
if $H^{1,0}(A)$ is a free $E\otimes_{\mathbb{Q}{}}\mathbb{C}{}$-module. In
this case, $d\overset{\df}{=}\dim_{E}H^{1}(A,\mathbb{Q}{})$ is even and the
subspace $W_{E}(A)\overset{\df}{=}\bigwedge\nolimits_{E}^{d}H^{1}%
(A,\mathbb{Q}{})$ of $H^{d}(A,\mathbb{Q}{})$ consists of Hodge classes
(\cite{deligne1982}, 4.4). When $E$ has degree $2$ over $\mathbb{Q}$, these
Hodge classes were studied by Weil (1977)\nocite{weil1977}, and for this
reason are called Weil classes. A polarization of $(A,\nu)$ is a polarization
$\lambda$ of $A$ whose Rosati involution stabilizes $\nu(E)$ and acts on it as
complex conjugation. The Riemann form of such a polarization can be written%
\[
(x,y)\mapsto\Tr_{E/\mathbb{Q}{}}(f\phi(x,y))
\]
for some totally imaginary element $f$ of $E$ and $E$-hermitian form $\phi$ on
$H_{1}(A,\mathbb{Q}{})$. If $\lambda$ can be chosen so that $\phi$ is split
(i.e., admits a totally isotropic subspace of dimension $d/2$), then $(A,\nu)$
is said to be of split Weil type.
\end{plain}

\begin{plain}
\label{r3}(Deligne 1982, \S 5.) Let $F$ be a CM-algebra, let $\phi_{1}%
,\ldots,\phi_{2p}$ be CM-types on $F$, and let $A=\prod\nolimits_{i}A_{i}$,
where $A_{i}$ is an abelian variety of CM-type $(F,\phi_{i})$. If $\sum
_{i}\phi_{i}(s)=p$ for all $s\in T\overset{\df}{=}\Hom(F,\mathbb{\mathbb{Q}{}%
}^{\mathrm{al}}{})$, then $A$, equipped with the diagonal action of $F$, is of
split Weil type. Let $I=\{1,\ldots,2p\}$ and $H^{r}(A)=H^{r}(A,\mathbb{Q}%
{}^{\mathrm{al}})$. In this case, there is a diagram
\[
\begin{tikzcd}[column sep=tiny+]
W_{F}(A)\otimes\mathbb{\mathbb{Q}}^{\mathrm{al}}\arrow[equals]{r}{\text{def}}
& \Big(\dstyle\bigwedge\nolimits_{F}^{2p}H^{1}(A,\mathbb{Q}{})\Big)\otimes_{\mathbb{Q}}\mathbb{Q}%
^{\mathrm{al}}\arrow[equals]{d}\arrow[hook]{r}
&\Big(\dstyle\bigwedge\nolimits_{\mathbb{Q}}^{2p}H^{1}(A,\mathbb{Q}{})\Big)\otimes_{\mathbb{Q}{}}\mathbb{Q}^{\mathrm{al}}
\arrow[equals]{d}\arrow[equals]{r}
&H^{2p}(A)\\
&\dstyle\bigoplus_{t\in T}\Big(\bigotimes_{i\in I}H^{1}(A_{i}%
)_{t}\Big)\arrow[hook]{r}
&\dstyle\bigoplus_{\substack{J\subset I\times T\\|J|=2p\\}}\Big(\bigotimes_{(i,t)\in J}H^{1}(A_{i})_{t}\Big)
\end{tikzcd}
\]

\end{plain}

\section{Theorem 1.}

\begin{theorem}
[\cite{andre1992}]\label{r0} Let $A$ be a complex abelian variety of CM-type.
There exist abelian varieties $A_{\Delta}$ and homomorphisms $f_{\Delta}\colon
A\rightarrow A_{\Delta}$ such that every Hodge class $t$ on $A$ can be written
as a sum $t=\sum f_{\Delta}^{\ast}(t_{\Delta})$ with $t_{\Delta}$ a Weil class
on $A_{\Delta}$.
\end{theorem}

\begin{proof}
Let $A$ be of CM-type and let $p\in\mathbb{N}{}$. We may suppose that $A$ is a
product of simple abelian varieties $A_{i}$ and let $E=\prod_{i}\End^{0}%
(A_{i})$. Then $E$ is a CM-algebra, and $A$ is of CM-type $(E,\phi)$ for some
CM-type $\phi$ on $E$. Let $F$ be a CM subfield of $\mathbb{C}{}$, Galois over
$\mathbb{Q}$, splitting the centre of $\End^{0}(A)$. Then $F$ splits $E$. We
shall show that Theorem 1 holds with each $A_{\Delta}$ of split Weil type
relative to $F$. Let $T=\Hom(F,\mathbb{Q}{}^{\mathrm{al}})$, where
$\mathbb{Q}{}^{\mathrm{al}}$ is the algebraic closure of $\mathbb{Q}{}$ in
$\mathbb{C}{}$. As $F\subset\mathbb{Q}{}^{\mathrm{al}}$, we can identify $T$
with $\Gal(F/\mathbb{Q}{})$.

Fix a subset $\Delta$ of $S\overset{\df}{=}\Hom(E,F)$ satisfying (\ref{eq2}).
For $s\in\Delta$, let $A_{s}=A\otimes_{E,s}F$. Then $A_{s}$ is an abelian
variety of CM type $(F,\phi_{s})$, where $\phi_{s}(t)=\phi(t\circ s)$ for
$t\in T$. Because $\Delta$ satisfies (\ref{eq2}),%
\[
\sum\nolimits_{s\in\Delta}\phi_{s}(t)\overset{\df}{=}\sum\nolimits_{s\in
\Delta}\phi(t\circ s)=p\text{, all }t\in T,
\]
and so we can apply \ref{r3}: the abelian variety $A_{\Delta}\overset{\df}{=}%
\prod\nolimits_{s\in\Delta}A_{s}$ equipped with the diagonal action of $F$ is
of split Weil type. There is a homomorphism $f_{\Delta}\colon A\rightarrow
A_{\Delta}$ such that
\[
f_{\Delta\ast}\colon H_{1}(A,\mathbb{Q}{})\rightarrow H_{1}(A_{\Delta
},\mathbb{Q}{})\simeq H_{1}(A,\mathbb{Q})\otimes_{E}F^{\Delta}%
\]
is $x\mapsto x\otimes1$. Here $F^{\Delta}$ is a product of copies of $F$
indexed by $\Delta$. The map $f_{\Delta}^{\ast}\colon H^{1}(A_{\Delta
},\mathbb{Q}{})\rightarrow H^{1}(A,\mathbb{Q}{})$ is the $E$-linear dual of
$f_{\Delta\ast}$.

Note that $A_{\Delta}$ has complex multiplication by $F^{\Delta}$. According
to \ref{r4}(a),
\[
H^{2p}(A_{\Delta})\overset{\df}{=}H^{2p}(A_{\Delta},\mathbb{Q}{}^{\mathrm{al}%
})=\bigoplus\nolimits_{J}H^{2p}(A_{\Delta})_{J},\quad H^{2p}(A_{\Delta}%
)_{J}\overset{\df}{=}\bigotimes\nolimits_{(s,t)\in J}H^{1}(A_{s})_{t},
\]
where $J$ runs over the subsets of $\Delta\times T$ of size $2p$. Let
$W_{F}(A_{\Delta})\subset H^{2p}(A_{\Delta},\mathbb{Q}{})$ be the space of
Weil classes on $A_{\Delta}$. Then $W_{F}(A_{\Delta})\otimes\mathbb{Q}%
{}^{\mathrm{al}}=\bigoplus\nolimits_{t\in T}H^{2p}(A_{\Delta})_{\Delta
\times\{t\}}$. Note that $a\in E$ acts on $H^{2p}(A_{\Delta})_{\Delta
\times\{t\}}\overset{\df}{=}\bigotimes\nolimits_{s\in\Delta}H^{1}(A_{s})_{t}$
as multiplication by $\prod\nolimits_{s\in\Delta}(t\circ s)(a)$. Therefore,
$f_{\Delta}^{\ast}\otimes1\colon H^{2p}(A_{\Delta})\rightarrow H^{2p}(A)$ maps
$H^{2p}(A_{\Delta})_{\Delta\times\{t\}}$ into $H^{2p}(A)_{t\circ\Delta}\subset
B^{p}(A)\otimes\mathbb{Q}{}^{\mathrm{al}}$.

In summary: for every subset $\Delta$ of $S$ satisfying (\ref{eq2}), we have a
homomorphism\linebreak$f_{\Delta}\colon A\rightarrow A_{\Delta}$ from $A$ into
an abelian variety $A_{\Delta}$ of split Weil type relative to $F$; moreover,
$f_{\Delta}^{\ast}(W_{F}(A_{\Delta}))\otimes\mathbb{Q}^{\mathrm{al}}$ is
contained in $B^{p}(A)\otimes\mathbb{Q}{}^{\mathrm{al}}$ and contains
$H^{2p}(A)_{\Delta}$. As the subspaces $H^{2p}(A)_{\Delta}$ span $B^{p}%
\otimes\mathbb{Q}{}^{\mathrm{al}}$ (see \ref{r4}(c)), this implies that the
subspaces $f_{\Delta}^{\ast}(W_{F}(A_{\Delta}))$ span $B^{p}$.\footnote{Let
$W$ and $W^{\prime}$ be subspaces of a $k$-vector space $V$, and let $K$ be a
field containing $k$. If $W\otimes_{k}K\subset W^{\prime}\otimes_{k}K$, then
$W\subset W^{\prime}$. Indeed, if $W$ is not contained in $W^{\prime}$, then
$(W+W^{\prime})/W^{\prime}\neq0$; but then $(W\otimes K+W^{\prime}\otimes
K)/W\otimes K\neq0$, which implies that $W\otimes K$ is not contained in
$W^{\prime}\otimes K$.}
\end{proof}

\section{Deligne's original version of Theorem 1.}

Let $E$ be a CM-field Galois over $\mathbb{Q}{}$, and let $\mathcal{S}{}$ be
the set of CM-types on $E$. For each $\Phi\in\mathcal{S}{}$ choose an abelian
variety $A_{\Phi}$ of CM-type $(E,\Phi)$, and let $A_{\mathcal{S}{}}%
=\prod\nolimits_{\Phi\in\mathcal{S}{}}A_{\Phi}$. Define $G^{H}$ (resp. $G^{W}%
$) to be the algebraic subgroup of $\GL_{H^{1}(A_{\mathcal{S}{}},\mathbb{Q}%
{})}$ fixing all Hodge classes (resp. divisor classes and split Weil classes)
on all products of powers of the $A_{\Phi}$, $\Phi\in\mathcal{S}{}$.

\begin{theorem}
\label{t2}The algebraic groups $G^{H}$ and $G^{W}$ are equal.
\end{theorem}

\begin{proof}
As divisor classes and Weil classes are Hodge classes, certainly $G^{H}\subset
G^{W}$. On the other hand, the graphs of homomorphisms of abelian varieties
are in the $\mathbb{Q}{}$-algebra generated by divisor classes
(\cite{milne1999lc}, 5.6), and so, with the notation of Theorem 1, $G^{W}$
fixes the elements of $f_{\Delta}^{\ast}(W_{F}(A_{\Delta}))$. As these span
the Hodge classes, we deduce that $G^{W}\subset G^{H}$.
\end{proof}

\begin{remark}
\label{r7}Theorem \ref{t2} is Deligne's original theorem (1982, \S 5) except
that, instead of requiring $G^{W}$ to fix all divisor classes, he requires it
to fix certain specific homomorphisms.
\end{remark}

\section{Applications}

\begin{plain}
\label{r10}Call a rational cohomology class $c$ on a smooth projective complex
variety $X$ \emph{accessible }if it belongs to the smallest family of rational
cohomology classes such that,

\begin{enumerate}
\item the cohomology class of every algebraic cycle is accessible;

\item the pull-back by a map of varieties of an accessible class is accessible;

\item if $(X_{s})_{s\in S}$ is an algebraic family of smooth projective
varieties with $S$ connected and smooth and $(t_{s})_{s\in S}$ is a family of
rational classes (i.e., a global section of $R^{r}f_{\ast}\mathbb{Q}{}\ldots)$
such that $t_{s}$ is accessible for one $s$, then $t_{s}$ is accessible for
all $s$.
\end{enumerate}

\noindent Accessible classes are automatically Hodge, even absolutely Hodge
(\cite{deligne1982}, \S \S 2,3).
\end{plain}

\begin{theorem}
\label{t3}For abelian varieties, every Hodge class is accessible.
\end{theorem}

\begin{proof}
This is proved in \cite{deligne1982} (see its Introduction) except that the
statement there includes an extra \textquotedblleft
tannakian\textquotedblright\ condition on the accessible classes (ibid.,
p.~10, (c)). However, this condition is used only in the proof that Hodge
classes on CM abelian varieties are accessible. Conditions (a) and (c) of
\ref{r10} imply that split Weil classes are accessible (ibid., 4.8), and so
this follows from Theorem 1 (using \ref{r10}(b)).
\end{proof}

\begin{remark}
\label{r5}In particular, we see that the Hodge conjecture holds for abelian
varieties if algebraic classes satisfy the variational Hodge conjecture (i.e.,
condition \ref{r10}(c)).
\end{remark}

\begin{remark}
\label{r6}For Theorem 3, it suffices to assume that \ref{r10}(c) holds for
families of abelian varieties over a \emph{complete smooth curve }$S$. Indeed,
(c) is used in the proof of the theorem only for families of abelian varieties
$(A_{s})_{s\in\mathbb{S}}$ with additional structure over a locally symmetric
variety $S$. More precisely, there is a semisimple algebraic group $G$ over
$\mathbb{Q}{}$, a bounded symmetric domain $X$ on which $G(\mathbb{R}{})$ acts
transitively with finite kernel, and a congruence subgroup $\Gamma\subset
G(\mathbb{Q}{})$ such that $S(\mathbb{C}{})=\Gamma\backslash X$
(\cite{deligne1982}, proofs of 4.8, 6.1). For $s\in S(\mathbb{C}{})$, the
points $s^{\prime}$ of the orbit $G(\mathbb{Q}{})\cdot s$ are dense in $S$ and
each abelian variety $A_{s^{\prime}}$ is isogenous to $A_{s}$. The boundary of
$S$ in its minimal (Baily-Borel) compactification has codimension $\geq2$.
After Bertini, for any pair of points $s_{1},s_{2}\in S(\mathbb{C}{})$, we can
find a smooth linear section of $S$ meeting both orbits $G(\mathbb{Q}{})\cdot
s_{1}$ and $G(\mathbb{Q}{})\cdot s_{2}$ but not meeting the boundary. This
proves what we want. Cf. \cite{andre1996}, p.~32.
\end{remark}

Let $X$ be an algebraic variety of dimension $d$, and let $L\colon H^{\ast
}(X,\mathbb{Q}{})\rightarrow H^{\ast+2}(X,\mathbb{Q}{})$ be the Lefschetz
operator defined by a hyperplane section of $X$. The strong Lefschetz theorem
says that $L^{d-i}\colon H^{i}(X,\mathbb{Q})\rightarrow H^{2d-i}%
(X,\mathbb{Q}{})$ is an isomorphism for all $i\leq d$. Let $a\!H^{2i}%
(X,\mathbb{Q}{})$ denote the $\mathbb{Q}$-subspace of $H^{2i}(X,\mathbb{Q}{})$
spanned by the algebraic classes. Then $L^{d-2i}$ induces an injective map
$L^{d-2i}\colon a\!H^{2i}(X,\mathbb{Q})\rightarrow a\!H^{2d-2i}(X,\mathbb{Q}%
{})$. The standard conjecture of Lefschetz type asserts that this map is
surjective for all $i\leq d$. It is known to be true for abelian varieties.

\begin{proposition}
[\cite{abdulali1994}, \textnf{p.~1122}]\label{p1}Let $f\colon A\rightarrow S$
be an abelian scheme over a smooth complete complex variety $S$. Assume that
the Lefschetz standard conjecture holds for $A$. Let $t$ be a global section
of the sheaf $R^{2r}f_{\ast}\mathbb{Q}(r)$; if $t_{s}\in H^{2r}(A_{s}%
,\mathbb{Q}{}(r))$ is algebraic for one $s\in S(\mathbb{C}{})$, then it is
algebraic for all $s$.
\end{proposition}

\begin{proof}
For $n\in\mathbb{N}{}$, let $\theta_{n}$ denote the endomorphism of $A/S$
acting as multiplication by $n$ on the fibres. By a standard argument
(\cite{kleiman1968}, p.~374), $\theta_{n}^{\ast}$ acts as $n^{j}$ on
$R^{j}f_{\ast}\mathbb{Q}{}$. As $\theta_{n}^{\ast}$ commutes with the
differentials $d_{2}$ of the Leray spectral sequence $H^{i}(S,R^{j}f_{\ast
}\mathbb{Q}{})\implies H^{i+j}(A,\mathbb{Q}{})$, we see that the spectral
sequence degenerates at the $E_{2}$-term and
\[
H^{r}(A,\mathbb{Q}{})\simeq\bigoplus\nolimits_{i+j=r}H^{i}(S,R^{j}f_{\ast
}\mathbb{Q}{})
\]
with $H^{i}(S,R^{j}f_{\ast}\mathbb{Q}{})$ the subspace of $H^{i+j}%
(A,\mathbb{Q}{})$ on which $\theta_{n}$ acts as $n^{j}$. Let $s\in
S(\mathbb{C}{})$ and $\pi=\pi_{1}(S,s)$. The inclusion $j_{s}\colon
A_{s}\hookrightarrow A$ induces an isomorphism $j_{s}^{\ast}\colon
H^{0}(S,R^{2r}f_{\ast}\mathbb{Q}{})\hookrightarrow H^{2r}(A_{s},\mathbb{Q}%
{})^{\pi}$ preserving algebraic classes, and so%
\begin{equation}
\dim a\!H^{0}(S,R^{2r}f_{\ast}\mathbb{Q}{})\leq\dim a\!H^{2r}(A_{s}%
,\mathbb{Q}{})^{\pi}. \label{e2}%
\end{equation}
Similarly, the Gysin map $j_{s\ast}\colon H^{2d-2r}(A_{s},\mathbb{Q}%
{})\rightarrow{}H^{2d-2r+2m}(A,\mathbb{Q}{})$, where $m=\dim(S)$ and
$d=\dim(A/S)$, induces a map $H^{2d-2r}(A_{s},\mathbb{Q}{})^{\pi}\rightarrow
H^{2m}(S,R^{2d-2r}f_{\ast}\mathbb{Q})$ preserving algebraic classes, and so%
\begin{equation}
\dim a\!H^{2d-2r}(A_{s},\mathbb{Q}{})^{\pi}\leq\dim a\!H^{2m}(S,R^{2d-2r}%
f_{\ast}\mathbb{Q})\text{.} \label{e3}%
\end{equation}
Because the Lefschetz standard conjecture holds for $A_{s}$,%
\begin{equation}
\dim a\!H^{2r}(A_{s},\mathbb{Q}{})^{\pi}=\dim a\!H^{2d-2r}(A_{s},\mathbb{Q}%
{})^{\pi}\text{.} \label{e4}%
\end{equation}
Hence,
\begin{align*}
\dim a\!H^{0}(S,R^{2r}f_{\ast}\mathbb{Q}{})  &  \overset{(\text{\ref{e2}%
})}{\leq}\dim a\!H^{2r}(A_{s}{},\mathbb{Q}{})^{\pi}\overset{(\text{\ref{e4}%
)}}{=}\dim a\!H^{2d-2r}(A_{s},\mathbb{Q}{})^{\pi}\\
&  \overset{(\text{\ref{e3}})}{\leq}\dim a\!H^{2m}(S,R^{2d-2r}f_{\ast
}\mathbb{Q})\text{.}%
\end{align*}
The Lefschetz standard conjecture for $A$ implies that
\[
\dim a\!H^{0}(S,R^{2r}f_{\ast}\mathbb{Q}{})=\dim a\!H^{2m}(S,R^{2d-2r}f_{\ast
}\mathbb{Q}),
\]
and so the inequalities are equalities. Thus%
\[
a\!H^{2r}(A_{s},\mathbb{Q}{})^{\pi}=a\!H^{0}(S,R^{2r}f_{\ast}\mathbb{Q}),
\]
which is independent of $s$.
\end{proof}

\begin{theorem}
[Abdulali, Andr\'{e}]\label{t4}The Lefschetz standard conjecture for algebraic
varieties over $\mathbb{C}$ implies the Hodge conjecture for abelian varieties.
\end{theorem}

\begin{proof}
By Theorem \ref{t3}, it suffices to show that algebraic classes are
accessible. They obviously satisfy conditions (a) and (b) of \ref{r10}, and it
suffices to check (c) with $S$ a complete smooth curve (Remark \ref{r6}). This
Proposition \ref{p1} does.
\end{proof}

\begin{remark}
\label{r8}Proposition 1 applies also to absolute Hodge classes and motivated
classes. As these satisfy the Lefschetz standard conjecture, we deduce that
Hodge classes on abelian varieties are absolutely Hodge and motivated.
\end{remark}

\begin{remark}
\label{r9}Let $H_{B}$ denote the Betti cohomology theory and $\Mot_{H}%
(\mathbb{C}{})$ the category of motives over $\mathbb{C}{}$ for homological
equivalence generated by the algebraic varieties for which the K\"{u}nneth
projectors are algebraic. If the Betti fibre functor $\omega_{B}$ on
$\Mot_{H}(\mathbb{C})$ is conservative, then the Lefschetz standard conjecture
holds${}$ for the varieties in question.

Indeed, as $L^{d-2i}\colon H_{B}^{2i}(X)(i)\rightarrow H_{B}^{2d-2i}%
(X,\mathbb{Q}{})(d-i)$ is an isomorphism (strong Lefschetz theorem), so also
is $l^{d-2i}\colon h^{2i}(X)(i)\rightarrow h^{2d-2i}(X)(d-i)$ by our
assumption on $\omega_{B}$. When we apply the functor $\Hom(\1,-)$ to this
last isomorphism, it becomes $L^{d-2i}\colon a\!H_{B}^{2i}(X)(i)\rightarrow
a\!H_{B}^{2d-2i}(X)(d-i)$, which is therefore an isomorphism. Thus the
standard conjecture $A(X,L)$ is true, and $A(X\times X,L\otimes1+1\otimes L)$
implies $B(X).$
\end{remark}

\bibliographystyle{cbe}
\bibliography{D:/Current/refs}
\end{document}